\documentclass[12pt]{article}
\usepackage{amsthm}
\usepackage{graphicx}
\usepackage{caption}
\usepackage{subcaption}
\usepackage{amssymb}
\usepackage{amsmath}
\usepackage{xfrac}
\usepackage{color}
\usepackage{hyperref}

\newtheorem*{definition}{Definition}

\newtheorem*{theorem*}{Theorem}
\newtheorem{theorem}{Theorem}
\newtheorem{lemma}[theorem]{Lemma}

\newtheorem{corollary}[theorem]{Corollary}

\newtheorem{observation}[theorem]{\textbf{Observation}}

\newtheorem{conjecture}[theorem]{Conjecture}

\newcommand\cF{\mathcal F}

\newcommand\cQ{\mathcal Q}

\newcommand\cW{\mathcal W}

\DeclareMathOperator{\bo}{O}
\DeclareMathOperator{\wpn}{wpn}
\DeclareMathOperator{\FH}{Forb}

\newcommand\eps{\varepsilon}

\newcommand\wH{\wpn(H)}

\title{The Asymptotic $\chi$-Boundedness of Hereditary Families}

\author{
{\sl Bruce Reed}\thanks{Mathematical Institute, Academia Sinica, Taiwan. bruce.al.reed@gmail.com.  Supported by  NSTC Grant 112-2115-M-001 -013 -MY3}
 \and
{\sl Yelena Yuditsky}\thanks{Université libre de Bruxelles, Brussels, Belgium; \texttt{yuditskyL@gmail.com}.}
}

\begin{document}

\maketitle

\begin{abstract}

A family ${\cal F}$ of graphs is asymptotically $\chi$-bounded with bounding function $f$ if almost every graph $G$ in the family satisfies $\chi(G) \le f(\omega(G))$.
A graph is $H$-free if it does not contain $H$ as an induced subgraph. We ask which hereditary families are asymptotically $\chi$-bounded, and discuss some related questions.
We show that for every tree $T$, almost all $T$-free graphs $G$ satisfy $\chi(G)=\omega(G)$. We show that for every cycle $C_k$ except $C_6$, almost every $C_k$-free graph $G$ satisfies $\chi(G) = \omega(G)$. We show that the $C_6$-free graphs are asymptotically $\chi$-bounded with bounding function $f(w)=(1+o(1))\frac{w^2}{\log w}$.   
\end{abstract}

\section{Overview}

The clique number\footnote{For this and other standard notation and terminology see Bondy and Murty~\cite{BMbook}.}, $\omega(G)$, of a graph $G$ is a lower bound on its chromatic number, $\chi(G)$. On the other hand, there is no upper bound on the chromatic number which is solely a function of its clique number. 
Indeed, there are graphs of clique number two which have arbitrarily high chromatic number. The first construction of such graphs was given by Zykov \cite{Z49} in 1949. A  natural and well studied question  is to determine for which families ${\cal F}$,  $\chi(G)$ is bounded by  $\omega(G)$ or some function of it, for those $G$ in ${\cal F}$. A family ${\cal F}$ is called {\it $\chi$-bounded}  with {\it bounding function }$f$, if every $G \in {\cal F}$ satisfies $\chi(G) \le f(\omega(G))$.

Perhaps  the most renowned  $\chi$-bounded family is the perfect graphs.

\begin{definition}
A graph is $G$ is  perfect if every induced subgraph $F$ of $G$ satisfies: $\chi(F)=\omega(F)$.  
\end{definition}

\begin{theorem}
    (The Strong Perfect Graph Theorem, \cite{CRST02}) A graph is perfect if and only if  for all $k \ge 2$
    it contains no induced subgraph isomorphic to  the odd cycle $C_{2k+1}$ or its complement  $\overline{C_{2k+1}}$. 
\end{theorem}

We focus on families of graphs for which almost all\footnote{a property holds for almost all graphs in a family if the proportion of graph of size $n$ in the family for which it fails goes to zero as $n$ goes to infinity.} graphs in the family satisfy $\chi(G)  \le f(\omega(G))$ for some function $f$. We say such a family is {\it asymptotically $\chi$-bounded} with bounding function $f$. We discuss {\it hereditary} families ${\cal F}$. This means that ${\cal F}$ is closed under the taking of induced subgraphs or equivalently for some set $O_{\cal F}$ of obstructions, a graph is in ${\cal F}$ precisely if it has no obstruction in $O_{\cal F}$ as an induced subgraph. Our main results are the following:

\begin{definition}
A graph is $H$-free if it does not contain $H$ as an induced subgraph.
\end{definition}

\begin{theorem}\label{thm2}
Let $T$ be a tree, then almost all $T$-free graphs $G$ have $\chi(G)=\omega(G)$. 
\end{theorem}

\begin{theorem}\label{thm666}
For any $k \ne 6$, almost all $C_k$-free graphs $G$ have $\chi(G)=\omega(G)$. 
\end{theorem}

\begin{theorem}\label{thm667}
The $C_6$-free graphs $G$ are asymptotically $\chi$-bounded with  a bounding function $f(w)=(1+o(1))\frac{w^2}{\log w}$.
\end{theorem}

Our approach to proving Theorems \ref{thm2}, \ref{thm666}, and \ref{thm667}
builds on that taken by Pr{\"o}mel and Steger in \cite{PSBerge} to prove Theorem \ref{thm666} for $k \in \{4,5\}$. 

The key to their proof for $k=5$  is the following result:

\begin{theorem}\label{thm4}
The vertex set of almost every $C_5$-free graph can be partitioned into $V_1$ and $V_2$ such that either (i) $V_1$ induces a clique and $V_2$ induces a disjoint union of cliques, or (ii) $V_1$ is a stable set and $V_2$ is complete multi-partite.   
\end{theorem}

This result can be used to obtain the  following strengthening of Theorem \ref{thm666} for $k=5$.

\begin{corollary}\label{corperfect}
Almost all $C_5$-free graphs $G$ are perfect. 
\end{corollary}

\begin{proof}
For $k \ge 2$, when we delete a clique from $C_{2k+1}$,  we are left with a graph containing $P_3$ (in fact $P_{2k-1}$) as an induced subgraph. Thus, when we delete a stable set from $\overline{C_{2k+1}}$, we are left with a graph containing $\overline{P_3}$ as an induced subgraph. When we delete a stable set from $C_{2k+1}$ we  are left with a graph containing $\overline{P_3}$ as an induced subgraph. Thus when we delete a clique from $\overline{C_{2k+1}}$ we are left with a graph containing $P_3$ as an induced subgraph. Hence, if a graph has a partition satisfying either of (i) or (ii) in Theorem \ref{thm4}, it cannot contain $C_{2k+1}$ or $\overline{C_{2k+1}}$ as an induced subgraph. By the Strong Perfect Graph Theorem, we are done.
\end{proof} 

We remark that since the Strong Perfect Graph Theorem had not proven when Pr{\"o}mel and Steger proved their result,  they provided 
a slightly longer but still straightforward proof that graphs whose vertex set can be partitioned into a clique 
and the disjoint union of cliques is perfect.

Partitioning results similar to Theorem \ref{thm4} will be useful in proving Theorems \ref{thm2} and \ref{thm666}.
However, we need to develop new techniques for moving from a partition to $\chi$-boundedness, as the ad-hoc argument used in the proof of 
Corollary \ref{corperfect} is not really of general application.  
In Section \ref{hereditary}, we discuss  the partitioning results and how they might be exploited. In Section \ref{trees}, we use this approach to prove Theorems \ref{thm2} and \ref{thm666}. 
In Section \ref{other} we discuss how these asymptotic colourings can be extended to other hereditary families of graphs. 

In the remainder of this section, we explain why  $H$-free graphs for $H$ a tree or a cycle  are of particular interest. 
Before doing so we remark that Erd\H{o}s,  Kleitman and Rothschild \cite{EKR76} proved that almost every triangle free graph is bipartite and hence perfect while 
Pr{\"o}mel and Steger \cite{PS91} proved that almost every $C_4$ free graph is the disjoint union of a clique and a stable set and hence perfect.

\subsection{Why Trees and Cycles?}
To motivate our focus on trees and cycles, we briefly survey some key results on the $\chi$-boundedness of hereditary families. 

If ${\cal F}$ is a hereditary family whose members satisfy $\chi=\omega$ then they satisfy 
a stronger property, all of their induced subgraphs $H$ satisfy $\chi(H)=\omega(H)$. 
In 1961 \cite{B63}, motivated by a question of Shannon about the zero-error transmission rate of a noisy channel,  
Berge initiated the study of graphs with this property. He called such graphs {\it perfect}. 

Berge noted that for any $k \ge 2$, the cycle $C_{2k+1}$ satisfies $\omega(C_{2k+1})=2$
and $\chi(C_{2k+1})=3$. He also noted that  $\omega(\overline{C_{2k+1}})=k$ and 
$\chi(\overline{C_{2k+1}})=k+1$. Thus if a graph is to be perfect it can contain neither 
a cycle of odd length greater than three nor the complement of such a cycle as an induced subgraph.
Berge conjectured that a graph is perfect precisely if this property holds. This celebrated conjecture, known as the Strong Perfect Graph Conjecture, was proven in 2002 by Chudnovsky, Robertson, Seymour, and Thomas \cite{CRST02}.  

Motivated by the Strong Perfect Graph Conjecture, in 1987 Gy\'{a}rf\'{a}s wrote a paper entitled {\it Problems from the World Surrounding 
Perfect Graphs} \cite{G87}. In this paper he asked which hereditary families of graphs were $\chi$-bounded. 
He conjectured that the family of graphs containing no odd cycle of length at least five as an induced subgraph is 
$\chi$-bounded. He then went on to make further conjectures about the $\chi$-boundedness of graphs obtained by forbidding a set of cycles as induced subgraphs. 
He recalled a result of Erd\H{o}s, which states that for every $k \ge 3$ and $c$ there are graphs with chromatic 
number $c$ which do not contain a cycle of length less than $k$. We note that any such graph has clique number  
$2$. Thus, if ${\cal F}$ is obtained by forbidding  a finite set of cycles as subgraphs, then ${\cal F}$ is not $\chi$-bounded. Furthermore, if the $H$-free graphs are to be a $\chi$-bounded family then $H$ must be a forest.

Gy\'{a}rf\'{a}s conjectured that for every forest $H$, the $H$-free graphs are $\chi$-bounded. 
Since every forest $H$ is an induced subgraph of a tree $T$ with $|V(H)|+1$ vertices, his conjecture
is equivalent to the following conjecture which was independently made by Sumner 6 years earlier: 

\begin{conjecture} [Gy\'{a}rf\'{a}s-Sumner conjecture \cite{G87,S81}] \label{ConjGS}
Let $H$ be a tree, then the $H$-free graphs are $\chi$-bounded. 
\end{conjecture}

This renowned conjecture has received considerable attention in the 40 years since it was made.  
Nevertheless, it remains open for most trees. The following are some trees for which the conjecture is known to be true. 
(i) Subdivided stars \cite{S97} (note that this class includes stars and paths, for which the conjecture was shown to be true in \cite{G87}), (ii) trees that can be obtained from a tree of radius two by subdividing once some edges incident to the root \cite{SS20} (this class generalizes trees with radius two studied in \cite{KP94}, and trees obtained from trees with radius two by making exactly one subdivision in every edge adjacent to the root studied in \cite{KZ04}), (iii) trees which are obtained from a star and a subdivision of a star by joining their centers with a path \cite{CSS19}, (iv) trees obtained by two disjoint paths by adding an edge between them \cite{S18}. 
See the survey paper \cite{SSSurvey}  for sketches of the proofs for some of those results .

Given this state of affairs, it is quite natural to state and prove Theorems \ref{thm2} and \ref{thm666}.

\section{Witnessing Partitions and their use in Colouring $H$-free Graphs}\label{hereditary}

In this section we discuss results similar to Theorem \ref{thm4}, for $H$-free graphs when $H$ is a tree or a cycle other than $C_6$. 

The {\it witnessing partition number} of a graph $H$, denoted $wpn(H)$ is the maximum $t$ such that there are $s$ and $c$ with 
$s+c=t$ such that  $H$ cannot be partitioned into $s$ stable sets and $c$ cliques. 

For any $k \ge 2$, the witnessing partition number of $C_{2k+1}$ is $k$, as it cannot be partitioned into $k$ cliques, but can be partitioned into $k+1$ cliques, $k$ cliques and a stable set (in fact vertex), three stable sets (one of which is a vertex), and two stable sets and a clique. Now  a  triangle cannot be partitioned into  two stable sets, 
a stable set of size three cannot be partitioned into two cliques, and neither $C_4$ nor its complement
can be partitioned into a clique and a stable set. It follows that if $H$  has witnessing partition number one 
both it and its complement are trees. Hence $H$ is subgraph of $P_4$ and thus, by the Strong Perfect Graph Theorem,  every $H$-free graph is perfect and satisfies $\chi=\omega$,
so we restrict our attention to $H$ with $wpn(H) \ge 2$.

A {\it witnessing partition} of $H$-freeness for a graph $G$ is a partition of $V(G)$ into $V_1,\ldots,V_{\wpn(H)}$ such that for any partition $W_1,\ldots,W_{wpn(H)}$ of $V(H)$, there is an $i\in [\wpn(H)]$ such that $H[W_i]$ is not an 
induced subgraph of $G[V_i]$. For any such partition, letting $J_i$ be the family of induced subgraphs of $H$
not appearing as induced subgraphs of $G[V_i]$ and ${\cal F}_i$ be the hereditary family whose obstruction set is $J_i$,
we see that  (i) $G[V_i] \in {\cal F}_i$ and (ii) there does not exist a partition of  $V(H)$ into $W_1,\ldots,W_{wpn(H)}$ such that $H[W_i] \in {\cal F}_i$ for each $i$. For any ${\cal F}_1, \ldots, {\cal F}_i$ for which  (i) and (ii) are true, we say the partition $V_1,\ldots,V_n$ is
{\it certified}  by ${\cal F}_1,\ldots,{\cal F}_{wpn(H)}$.

Now, every $H$-free $G$ has a  trivial witnessing partition into $wpn(H)-1$ empty sets and $G$ itself. 
We want to focus on  partitions where the parts of the partition are not only not empty but also are reasonably 
large. If they are of size at least 
$4^{|V(H)|}$, this  implies each parts contains  a clique or stable set of size at least $|V(H)|$. Now for each $V_i$
in such a partition, let  $s_i$  be the number of other parts of the partition which 
contain a stable set  of size $|V(H)|$. Then the remaining  $c_i=wpn(H)-1-s_i$ parts contain 
a clique of size $|V(H)|$. Since $H$ can be partitioned into $a$ cliques and $b$ 
stable sets for any $(a,b)$ in $(s_i+2,c_i),(s_i+1,c_i+1), (s_i,c_i+2)$, there must be induced subgraphs of 
$H$ which are not induced subgraphs of $G[V_i]$ of each of the following types: bipartite, split\footnote{ a graph is split if its
vertex set can be partitioned into a clique and a stable set}, complement of bipartite. 

We say an $H$-freeness witnessing partition $V_1,\ldots,V_{wpn(H)}$ is {\it nontrivial} if 
it is certified by an ${\cal F}_1,\ldots,{\cal F}_{wpn(H)}$ such that the obstruction set for each
${\cal F}_i$ contains a bipartite graph, a split graph, and the complement of a bipartite graph. 

  For example, as the proof of Corollary \ref{corperfect} shows, partitions satisfying (i) or (ii) of  Theorem \ref{thm4} 
are $C_5$-freeness witnessing. Furthermore, they are certified by one of (cliques, disjoint union of cliques) 
or (stable sets, complete multipartite graphs) and hence are  nontrivial.  

Clearly, if there is a  non-trivial witnessing partition of $H$-freeness for $G$ then $G$ is $H$-free. For many graphs $H$, almost all $H$-free graphs have a  nontrivial $H$-freeness witnessing partition. This is the case for cliques \cite{KPR87}, cycles \cite{BB11,R,KKOT15} and trees \cite{RY}. Reed \cite{R} conjectured that this property holds for all $H$, but Norin \cite{N20} showed that this is not the case. 
We remark that Reed  proved his conjecture modulo an exceptional set $V_0$ of size $n^{1-\eps}$ for $\eps=\eps(H)>0$. 

In  addition to proving the  conjecture for $H=C_5$, 
Pr{\"o}mel and Steger \cite{PS93} laid the groundwork for it. They  introduced the   notion of the  witnessing partition number of a graph\footnote{
although they  did not define this term working instead with $\tau (H)-1$ where $\tau(H)$ is an invariant equal to  $\wpn(H)+1$. In the literature, this invariant is   sometimes referred to as either the binary chromatic (in short, bichromaric) number or the colouring number of a graph (see e.g. \cite{BT97})}, and highlighted the important role played by  
witnessing partitions  in characterizing the structure of typical $H$-free graphs \cite{PS92,PS93}.
We can exploit these partitions by applying  the following observation, 
with $V_0$ empty.  

\begin{observation} \label{splitobs1}
If for some monotone functions $f_1,\ldots,f_k$, $V(G)$ has a  partition $\{V_1,\ldots,V_k\}$  such that for every 
$i \ge 1$ we have $\chi(G[V_i]) \le f_i(\omega(G[V_i]))$ then $$\chi(G) \le \sum_{i=1}^kf(\omega(G)) \le k \max_i f_i(\omega(G)).$$
\end{observation}

We will use Observation \ref{splitobs1}  to prove Theorem \ref{thm667}, by combining it with the following result:

\begin{theorem}[RS]\label{reedscottC6}
Almost every $C_6$-free graph can be partitioned into a stable set  and the complement of a graph of girth five.
\end{theorem}

For our other two theorems, we cannot permit an extra factor of  $k$ to creep into the binding functions so we need slightly stronger results about $H$-free graphs as we now explain. 

For the remainder of the paper we will focus on the graphs on vertex set $[n]=\{1,2,...,n\}$. 
We let $Forb(H)$ be the set of graphs which are $H$-free. For any hereditary family ${\cal F}$,
we let ${\cal F}_n$ be the graphs in ${\cal F}$ with vertex set ${\cal V}_n=[n]$. 

By  a  {\it pattern} on a partition $V_1,V_2,\ldots,V_k$ a of $[n]$
we mean an ordered set $\{G_1,...,G_k\}$ where  $V(G_i)=V_i$ for each $i\in [k]$.
By an extension of the pattern we mean a  graph obtained by choosing 
the edges between vertices lying in distinct parts. 
For a sequence of families $\cF_1,\cF_2,\ldots,\cF_k$, a pattern  is a   $\cF_1,\cF_2,\ldots,\cF_k$-{\it pattern}
 if $G_i \in \cF_i$ for each $i \in [k]$. 

We prove our results by considering certain  partitions  of $[n]$ into $\wpn(H)$ sets, 
and a $H$-freeness witnessing  $\cF^*_1,\cF^*_2,\ldots,\cF^*_k$-pattern on such a partition certified by 
$\cF^*_1,\cF^*_2,\ldots,\cF^*_k$ for some bounded number of choices for $\cF^*_1,\cF^*_2,\ldots,\cF^*_k$.
Every extension of such a pattern is $H$-free, and we consider for each relevant partition every pair consisting of it and an extension of it.  
We  will show that  for the relevant families,  almost all 
extensions of  the  certified patterns we consider  satisfy $\chi=\omega$.
So, the total number of graphs in the family  not satisfying $\chi=\omega$ is the sum of the number of graphs 
in the family which do not have a certifying partition of the type we are considering and a number which is 
of smaller order then the number of (pattern, extension) pairs we consider.  
In order to translate this into a result about almost all $H$-free 
graphs, we need to ensure that  almost every $H$-free graph permits a partition of the type we are considering and that  the number of  extensions of the  $H$-freeness witnessing partitions we are considering is of the 
same order as the number of $H$-free graphs. 

By a {\it $\mu$-eqi-partition} of ${\cal V}_n$ we mean a partition such that the maximum difference 
between the size of a part and the average size  is at most $n^{1-\mu}$. 

Let $\cF$ be a hereditary family of graphs. We say that $\cF$ is \textit{($\mu$,l)-partitionable} if the following holds:  there is a constant $k\in \mathbb{N}$ and for each $1 \le j \le l$ an ordered sequence of families $\cF^j_1,\cF^j_2,\ldots,\cF^j_k$ such that,
\begin{itemize}
\item[(a)] Almost every  $G$ in $\cF_n$ permits a  $\cF^j_1,\cF^j_2,\ldots,\cF^j_k$ pattern $G_1,G_2,...,G_k$, 
for some $j$ with $1\le j \le l$, on some $\mu$-eqi-partition $V_1,V_2,\ldots,V_k$ of $V(G)$. 
\item[(b)] For $1 \le j \le l$, the sum over all $\mu$-eqi-partitions $V_1,V_2,\ldots,V_k$  of  extensions of ordered  $\cF^j_1,\cF^j_2,\ldots,\cF^j_k$-patterns of graphs $G$  on $V_1,V_2,\ldots,V_k$  is $\bo(k! |\cF_n|)$.
\end{itemize}

For any  positive integer $b$ and monotonic function with $f(x) \ge x$, we say such a family is $(b,f)$-nicely  $(\mu,l)$-partitionable if we can choose the $l$ certifying families so that  either (i) almost every graph $F$  in every ${\cal F}^j_i$ is $f(\omega)$ colourable using 
colours classes of size at most $b$, or (ii) for every $j$, there is a $\cF^j_i$ which consists only of cliques, 
and for all $1 \le p \le k$  such that ${\cal F}^j_p$ does not consist only of cliques, almost every graph $F$  in $\cF^j_p$ has $|V(F)|^{1-\frac{\mu}{2}}$
disjoint stable sets of size two.

Since it has already been shown that almost all $H$-free graphs are perfect if $H=P_4$ or $H=C_k$ for 
$k \le 5$, combining the following results yields Theorems \ref{thm2} and \ref{thm666}. We  prove the first theorem, 
in Section \ref{secmain} and the second   in  Section \ref{trees}.   

\begin{theorem}
\label{main}
The family $\FH(H)$ is $(2,f)$-nicely $(\mu_H,2)$-partitionable for  $f$ the identity function and some $\mu_H>0$ for every  $H$ which is either a cycle of length at least 7  or a tree which is not a subgraph of $P_4$. 
\end{theorem}

\begin{theorem}
\label{theorem1}  For any $b,w>0$, the following holds. Suppose we are given  a  partition of $V_1,\ldots,V_w$ of $[n]$ and a $c$-colouring of each element of a pattern for it such that each colour class has at most $b$ vertices. 
Then the probability that an extension of the pattern does not have a $c$-colouring using colour classes of  size at most $2^{w-1}b$ is $o_c(1)$.
\end{theorem}

\begin{corollary}
\label{cormain}
For any $l,b'$ and monotonic $f$ with $f(x) \ge x$, if   $\cF$ is a $(b',f)$-nicely  $(\mu,l)$-partitionable hereditary property then almost all graphs $G$ in $\cF$ are $f(\omega(G))$-colourable. 
\end{corollary}

\begin{proof}
We consider the at most $l$ sequences of certifying families of graphs which show $\cF$ is 
a $(b',f)$-nicely  $(\mu,l)$-partitionable family.
We let $w$ be their common length. By the definition of $(\mu,l)$-partionable it is enough to 
show that for every $j$ between $1$ and $l$, for  every $\mu$-eqipartition of $[n]$,
for almost every ${\cal F}^j_1,\ldots,{\cal F}^j_w$-pattern $G_1,\ldots,G_w$ over the partition, almost every extension of the pattern 
is $f(\omega(G))$ colourable. Applying Theorem \ref{theorem1} with $b=\max(b',2)$  we need only show that for every $l$,
for almost every such pattern there is an $f(\omega(G))$ colouring of each $G_i$ using colours classes of size at 
most $b$.  
This is true for $j$ with $1 \le j \le l$ for which for  all $1 \le l \le w$ almost every graph $F$  in ${\cal F}^j_i$
has an $f(\omega(F))$ colouring with colour classes of size at most $b$. For any other $j$, we have that some $G_i$ is a clique and 
hence $\omega(G) \ge \frac{n}{w}-n^{1-\mu}$. Now, for every $1 \le k \le w$, either ${\cal F}^j_k$ consists only 
of cliques and hence $\chi(G_k)=\omega(G_k) \le f(|G|)$ and  $G_k$ has an $f(\omega(G))$ colouring using singletons or 
almost every element $F$   of ${\cal F}^j_k$ has a $|F|-n^{1-\frac{\mu}{2}}$ colouring using stable sets 
of size 2, and hence almost every choice for $G_k$ has an $\omega(G)$ colouring using stable sets of size at most 2. 
\end{proof}

\section{The Strange Case of $C_6$}\label{C6sec}

We now show how to use the structural results from the last section to prove Theorem \ref{thm667}.

\begin{proof} [Proof of Theorem \ref{thm667}]
By Theorem \ref{reedscottC6} almost all $G\in \FH_n(C_6)$ can be eqi-partitioned into $(V_1,V_2)$ so $G[V_1]$ is a stable set and $G[V_2]$ is a complement of a graph of girth $5$. 

A result of Shearer~\cite{Sh} states that every $l$ vertex  $E_3$-free graph,    contains  
a clique of size $(1+o(1))\sqrt{l \log l}$ and hence so does every complement of a graph of girth 
5 with $l$ vertices. Furthermore, if such a graph has no $\frac{l}{2}-\frac{l}{\log l}$ disjoint stable sets of size 2,
then it contains a clique of size at least $\frac{2l}{\log l}$. It follows for  the complement
 $F$ of a graph of girth 5, $\chi (F) \le (1+o(1))\frac{\omega(F)^2}{\log \omega(F)}$. 

Since we need one colour to colour a stable set, we are done. 
\end{proof}

We remark that if $F$ is the complement of a graph of girth 5, then it contains no stable set of size three
and hence $\chi(F) \ge \frac{|V(F)|}{2}$. So we know that almost every $C_6$-free graph $G$  on $n$ vertices 
has chromatic number exceeding  $\frac{n}{5}$. So, to improve the bound in Theorem \ref{thm667},  we need to show 
that for almost all  complements of graphs   of girth 5, $\omega(G)=\omega(\sqrt{|V(G)|\log |V(G)|})$.

\section{The Proof of Theorem \ref{main}}
\label{secmain}

\begin{proof}[Proof of Theorem \ref{main}]

Let $T$ be a tree which is not a subgraph of $P_4$. 
The following classes of trees are defined in  \cite{RY}: spiked, spiked stars, subdivided stars, and 
double stars.  A subdivided star is obtained by subdividing each edge of a star and hence has an odd number of vertices 
and no perfect matching. 

The following results are proven there: 
\begin{enumerate}
    \item If  $T$  has no perfect matching and is not a subdivided star, then  almost every $T$-free graph 
            can be partitioned into $wpn(H)=\alpha(T)-1$ cliques (\cite{RY}, Theorem 2.21),
     \item If  $T$  is a subdivided star, then  almost every $T$-free graph 
            can be partitioned into  either $wpn(H)=\alpha(T)-1$ cliques 
            or into $\alpha(T)-2$ cliques and a stable set, (\cite{RY}, Theorem 2.22),
     \item If  $T$  has  a perfect matching and is not a  spiked tree or a double star  then  almost every $T$-free graph 
            can be partitioned into $wpn(H)-1=\alpha(T)-2$ cliques  and a set inducing a graph whose complement only contains components which are triangles  or stars (\cite{RY}, Theorem 2.28),
        \item If  $T$  is a spiked tree which is not a spiked star then  almost every $T$-free graph 
            can be partitioned into $wpn(H)-2=\alpha(T)-3$ cliques  and  two sets both  inducing the complement 
            of a matching (\cite{RY}, Theorem 2.33),
            \item If  $T$  is a  spiked star then  almost every $T$-free graph 
            can be partitioned into $wpn(H)=\alpha(T)-1$  sets all of which induce the complement 
            of a matching (\cite{RY}, Theorem 2.35),
             \item If  $T$  is a  doublestar which is not $P_6$ then  almost every $T$-free graph 
            can be partitioned into $wpn(H)-1=\alpha(T)-2$  cliques  and a set inducing a graph whose complement only contains components which are cliques  or stars (\cite{RY}, Theorem 2.37),
            \item If  $T$  is  $P_6$ then $wpn(H)=2$  and almost every $T$-free graph 
            can be partitioned into   a clique  and a set inducing a graph whose complement only contains components which are 
            obtained from a clique and a stable set by adding all edges between them (\cite{RY}, Theorem 2.39).
\end{enumerate}

Now $P_4$ is self-complementary and connected. So, each of the parts in the partitions above induces 
a $P_4$-free graph. It is well-known that the number of $P_4$-free graphs on $n$ vertices is $O(2n^{2n})$. 
This follows from the fact that every such graph is disconnected or disconnected in the complement\cite{Se}
and so both the family of  connected $P_4$ free graphs  and the family of disconnected $P_4$-free 
graphs can be put into 1-1-correspondence with  the family of trees with $n$ 
leaves, each of whose internal nodes has at least 2 children, which contains only trees with at most $2n$
nodes. So, the number of choices of an ordered  partition as above and a pattern over it is $O(2^{4n\log n})$. 

The number of choices for a  graph extending any pattern over a partition $P$   is exactly  $2^{m_P}$ where 
$m_P$ is the number of pairs of vertices lying in different partition elements. If the partition is not an $\frac{1}{9}$-eqi-partition
then $2^{m_P}=2^{(1-\frac{1}{wpn(T)})n^2-\Omega(n^\frac{9}{5})}$. On the other hand there are 
are  more than  $2^{(1-\frac{1}{wpn(T)})n^2-2n}$  $H$-free graphs extending a partition of $V_n$ 
into $\wpn(H)$ parts each with size either  $\lfloor  \frac{n}{wpn(H)} \rfloor$ or $\lceil  \frac{n}{wpn(H)} \rceil$.
It follows that all the results above are true if we  further insist that the partition is a $\mu$-eqi-partition
for any $\mu \le \frac{1}{9}$.

Now, ordering the families  of the results above  so that   the first has the largest stability number, 
we see that for every $j$  and every  $i>1$, we have that the minimum degree of a graph 
in ${\cal F}^j_i$ on $l$ vertices  is at least $l-1$. It follows that we can apply the following

\begin{lemma} 
\label{nodoublecount}
(\cite{RY}, Lemma 2.20) Let $p\ge 1$ be fixed, and suppose that ${\mathcal F}_1,...,{\mathcal F}_p$ 
are families containing subgraphs of every size and  for sufficiently large $l_0$,  we have 
\begin{enumerate}
\item all the graphs  in each 
${\mathcal F}_i$  are $P_4$-free, and 
\item 
for $i>1$ and $l\ge l_0$, every graph in ${\mathcal F}_i$ on $l$ vertices  has minimum degree at least $\frac{31l}{32}+1$;
\end{enumerate}
Then  
for all $\mu>0$ and  sufficiently large $n$, the number of (graph, partition) pairs consisting of  
of a graph $G$ on $V_n$  and a $\mu$-eqi-partition of $V_n$ into  $X_1,...,X_p$ such that $G[X_i]$ is in ${\mathcal F}_i$ 
 for each $i$ is $\Theta(1)$ times the number of graphs on $V_n$ which  permit such a partition. 
\end{lemma}

Now, it is not hard to see that all the partitions in  Results (1)-(7)  are certifying, and hence all  the extensions of 
a certifying pattern are $T$-free. 
We see that for every tree $T$ which is not a subgraph of $P_4$,
the $T$-free graphs are a $(\frac{1}{9},2)$-partionable family.

Clearly, we can partition a stable set $S$ into $\frac{|S|}{2}$ stable sets of size 2.
A graph  $F$ which  is the complement   of a matching is perfect and has an $\omega(F)$
colouring using stable sets of size $2$. 

Now, we can specify a graph on $l$ vertices  all of whose components in the complement  are the join  of a clique 
and a stable set by specifying  the set of vertices which are in nonsingleton components of the complement,
the partition into components  in the complement of the remainder of the vertices, and  for each vertex whether it is in the clique or
stable set of its component in the complement. Any such graph which does not contain $l^{99/100}$ disjoint stable sets of size 2
contains at most $l^{99/100}$ nonsingleton components. Hence there are at most $2^{2l}(l^{99/100})^l$
such components.  Thus, Theorem 2.24 and Lemma 2.25  of \cite{RY} combine to yield that almost every graph  $J$  in ${\cal F}$ 
contains $|V(J)|^{99/100}$ disjoint stable sets of size $2$ when ${\cal F}$ is the  family whose components 
in the complements are restricted to be in any of the following sets:  (i) triangles and stars, (ii) cliques and stars, 
(iii) graphs obtained from a clique and a stable set by adding all edges between them.

We have shown that   for every $T$, the $T$-free graphs are $(2,f)$-nicely $(\frac{1}{9},2)$-partitionable
for $f$ the identity function. 

Balogh and Butterfield~\cite{BB11} have shown that for  $k$ odd and at least $9$,
almost every $C_k$ free graph has a partition into $wpn(C_k)=\frac{k-1}{2}$
cliques and almost every $C_7$ free graph either has a partition into 3 cliques
or into 2 cliques and a stable set. The same argument as above shows that 
we can strengthen these results by insisting that the partition be an eqi-partition
and that this implies odd cycles of length at least $7$ are $(2,f)$-nicely $(\frac{1}{9},2)$-partitionable
for $F$ the identity function.  

Reed~\cite{R} has shown that (i) for  $k$ even and at least $12$,
almost every $C_k$ free graph has a partition into $wpn(C_k)-1=\frac{k-4}{2}$
cliques and a set inducing a graph whose complements only has stars and triangles as components,
(ii) almost every $C_{10}$ free graph  has a partition into 3 cliques
and a set inducing a graph whose complements only has stars and cliques  as components,
and (iii) almost every $C_8$ free graph  has a partition into 2 cliques
and a set inducing a graph whose complements are formed from a clique and a stable set by adding 
all edges between them.  The same argument as above shows that 
we can strengthen these results by insisting that the partition be an eqi-partition
and that this implies even cycles of length at least $8$ are $(2,f)$-nicely $(\frac{1}{9},1)$-partitionable
for $F$ the identity function. 
\end{proof}

\section{The $H$-free graphs are asymptotically $\chi$-bounded when $H$ is a tree or a cycle}\label{trees}

We turn now to proving Theorem \ref{theorem1} (which, as mentioned earlier, completes the proof of  Theorems  \ref{thm2} and \ref{thm666}).
We need a few additional definitions and tools. For $G=(V,E)$  a graph and $V'\subseteq V$, we denote by $G[V']$ the subgraph of $G$ induced on $V'$. 
Let $v\in V$, we denote by $N_G(v)$ the set of vertices which are adjacent to $v$. For $V'\subset V$, we set  $N_G(V')=\cup_{v\in V'}N_G(v)$. If there is no confusion regarding the underlying graph, we do not write the subscript. 
Let $G$ be a bipartite graph with a bipartition $V(G)=A\cup B$, that is $G[A]$ and $G[B]$ are stable sets. We say that a matching $M$ in $G$ saturates $B$, if $B\subset \cup_{e\in M}e$.

\begin{theorem}[Hall's theorem \cite{BMbook}]\label{HallThm}
Let $G$ be a graph with a bipartition $V(G)=A\cup B$. There is a matching which saturates $B$ if and only if for every $W\subseteq B$, $|N(W)|\ge |W|$. 
\end{theorem}

\begin{theorem} [Chernoff Bound \cite{C81}] \label{Chernoff}
Let $X_1,X_2,..,X_n$ be independent random variables with $\mathbb{P}[X_i=1]=p_i$ and $\mathbb{P}[X_i=0]=1- p_i$. Let $X=\sum_{i=1}^n X_i$ with the expectation $\mathbb{E}[X]=\sum_{i=1}^np_i$. Then,
\begin{eqnarray*}
\mathbb{P} [X\le  E[X]-\lambda ] &\le& e^{-\frac{\lambda ^2}{2\mathbb{E}[X]}} \\
\mathbb{P} [X\ge E[X]+\lambda ] &\le& e^{-\frac{\lambda ^2}{2(\mathbb{E}[X]+\lambda/3)}}.
\end{eqnarray*}
\end{theorem}

\begin{proof}[Proof of Theorem  \ref{theorem1}]
We prove this by induction on $w$. When $w=1$ the theorem is trivially true.
Having proved it for $w=2$, we apply it to the graph induced by the union of the first two 
parts to obtain that  with probability $1-o_1(c)$ they have a $c$-colouring using parts of size at most $2b$.
We now  merge these two parts and apply the theorem  to our colour classes of size at most $2b$ in $w-1$ parts. 
So, it  is enough to prove the theorem when there are only two parts.  

We consider an auxiliary  bipartite graph where each vertex is a colour class in a part,
and edges join 
two colour classes  in different  parts if there are no edges  of the extension between them. 
We ensure there are $2c$ vertices on each side by possibly adding empty colour classes.  
Thus, we construct a random  bipartite graph where 
there are $c$ vertices in each part and two vertices are joined with probability at least $2^{-b^2}$. 

We are done if we can show that the probability there is no perfect matching is $o_1(c)$.
Applying Hall's Theorem we see that there is no perfect matching precisely 
if there is a set $S$ on one side  such that $|N(S)|<|S|$. 
Now, for this to occur either some vertex  has degree at most $\frac{c}{2^{b^2+1}}$
or we have $|S|>|N(S)| \ge \frac{c}{2^{b^2+1}}$. Since $|N(S)| <|S|  \le c$, this lower bound on the 
degree also implies that $c-|S| \ge \frac{c}{2^{b^2+1}}$

So, we need only show that (i) the probability that there is a vertex of degree at most $\frac{c}{2^{b^2+1}}$
is $o_c(1)$, and (ii) the probability that there are two sets of size at least $ \frac{c}{2^{b^2+1}}$ on opposite 
sides of the bipartition with no edges between them is $o_c(1)$.

The probability that (ii) fails for a specific pair of sets $(S_1,S_2) =2^{-|S_1|S_2|} \le  2^\frac{-c^2}{2^{2b^2+2}}$.
So the probability (ii) fails is at most $2^{2c-\frac{c^2}{2^{2b^2+2}}} =o_c(1)$.

The degree of a vertex is  the sum of $c$ $0-1$ variables each of which  is one with probability 
at least $2^{-b^2}$ and hence the probability it is less than some $d$ is at most the 
probability that the sum of $c$ $0-1$ variables each of which is one with probability 
exactly  $2^{-b^2}$ is less than $d$. Such a sum has expected value  $\frac{c}{2^{b^2}}$,
so if it is less than $\frac{c}{2^{b^2+1}}$ it differs from its expected value by at least $\frac{c}{2^{b^2+1}}$.
Applying the Chernoff Bounds we see that the probability that this occurs is at most $e^{-\frac{c}{4\cdot2^{b^2+1}}}=o(\frac{1}{c})$.
Hence the probability (i) holds is also $o_c(1)$
\end{proof}

\section{Asymptotic colouring of other hereditary families} \label{other}

The framework developed above can be applied to additional hereditary families of graphs. One of such families is the \textit{string graphs}. A string graph is the intersection graph of a family of continuous arcs in the plane. The following 
is the main result of \cite{PRY18}. 

\begin{theorem}[\cite{PRY18}]
A partition of a graph into 4 graphs, 3 of which are cliques and the fourth of which is the disjoint union of two 
cliques, certifies it is a string graph. Almost every string graph has such a certifying partition. 
\end{theorem}

This us allows to use the approach above to prove in exactly the same way  the following.

\begin{corollary}
Almost every string graph $G$ has $\chi(G)= \omega(G)$. 
\end{corollary}

Let $\cW(H)$ be the collection of all pairs $(s,c)$ such that $H$ cannot partitioned into $s$ stable sets and $c$ cliques where $s,c$ are such that $s+c=\wH$. Let $\cQ(H,s,c)$ be the set of all graphs that can be partitioned into $s$ stable sets and $c$ cliques. Let $\cQ(H)=\cup_{(s,c)\in\cW(H)}\cQ(H,s,c)$. Balogh and Butterfield\cite{BB11} defined the family of 
critical graphs, they  are those for which $\wH \ge 2$ and almost all $H$-free graphs are in $\cQ(H)$.

It is not hard to extend our framework to families $Forb(H)$ where $H$ is a critical graph and to deduce the following
unless one of the families consists of $wpn(H)$ stable sets.

\begin{corollary}
Let $H$ be a critical graph then almost every $H$-free graph $G$ has $\chi(G)= \omega(G)$. 
\end{corollary} 

If one of the families consist of all stable sets, then $H$-free graphs extending a pattern certified by this family 
are $wpn(H)$ colourable. It is not hard to see that almost all extensions contain a clique of size $wpn(H)$ 
with one vertex in each part, so the result holds in this case also. 

We note that Balogh and Butterfield provided an equivalent definition of critical, with the equivalence of the two definitions being the main result of their paper.  

Let $H$ be a graph and $s,c\in \mathbb{N}$, let ${\cal F}(H,s,c)$ denote the set of minimal (by induced containment) graphs $F$ such that $H$ can be covered by $s$ stable sets, $c$ cliques, and $F$.

\begin{theorem} [Balogh and Butterfield \cite{BB11}] \label{BBresult}
 A graph is critical precisely  if for all $s,c$ such that $s+c=\wH-1$ and large enough $n\in \mathbb{N}$, $|\left(\FH({\cal F}(H,s,c))\right)_n|\le 2$.
\end{theorem}

\end{document}